\documentclass[letterpaper,10pt]{amsart}

\usepackage{epsfig}
\usepackage{amssymb}
\usepackage{amsfonts}
\usepackage{amsmath}
\usepackage{graphicx}
\usepackage{mathrsfs}
\usepackage{color}
\usepackage{stmaryrd}
\usepackage{amsthm}
\usepackage[all]{xy}
\usepackage[top=1.25in, bottom=1.25in, left=1.25in, right=1.25in]{geometry}

\newtheorem{thm}{Theorem}[section]
\newtheorem{thmx}{Theorem}

\newtheorem{prop}{Proposition}

\newtheorem{conjecture}{Conjecture}

\newtheorem{rem}{Remark}

\begin{document}
\title{Gelfand problem and Hemisphere rigidity}

\author{Mijia Lai}
\address{School of Mathematical Sciences,
Shanghai Jiao Tong University}
\email{laimijia@sjtu.edu.cn}

\author{Wei Wei}
\address{Department of Mathematics, Nanjing University
}
\email{wei\_wei@nju.edu.cn}

\thanks{M. Lai's research is supported in part by National Natural Science Foundation of China No. 12031012, No. 12171313 and the Institute of Modern Analysis-A Frontier Research Center of Shanghai. W. Wei's research is supported in part by BoXin programme BX20190082.}
\begin{abstract}
We give an interpretation of the hemisphere rigidity theorem of Hang-Wang in the framework of Gelfand problem. More precisely, Hang-Wang showed that for a metric $g$ conformal to the standard metric $g_0$ on $S^{n}_{+}$ with $R\geq n(n-1)$ and whose boundary coincides with $g_0|_{\partial S^{n}_{+}}$, then $g=g_0$. This is related to the classical Gelfand problem, which investigates $-\Delta u=\lambda g(u)$ for certain nonlinearity $g$ in a bounded region $\Omega \subset \mathbb{R}^n$ subject to the Dirichlet boundary condition. It is well-known that there exists an extremal $\lambda^{*}$, such that for $\lambda>\lambda^{*}$, the above equation does not admit any solution. Interestingly, Hang-Wang's hemisphere rigidity theorem yields a precise value for $\lambda^{*}$ for $g(u)=e^{2u}$ when $n=2$ and $g(u)=(1+u)^{\frac{n+2}{n-2}}$ for $n\geq 3$. We attempt to generalize the hemisphere rigidity theorem under $Q$ curvature lower bound and fit this into the interpretation of fourth order Gelfand problem for bi-Laplacian with conformal nonlinearity.
\end{abstract}

\maketitle

%%%%%%
\section{Introduction}
There have been extensive studies of positive solutions to the semilinear equation 
\begin{align} \label{Gelfand2nd} \tag{$G_{\lambda}$}
\left\{
  \begin{array}{ll}
    -\Delta u= \lambda g(u) & \text{in $\Omega$}, \\
    u=0, & \text{on $\partial \Omega$}.
  \end{array}
\right.
\end{align}
Here $\Omega$ is a smooth bounded domain in $\mathbb{R}^n$, $\lambda$ is a nonnegative parameter and $g$ is a positive increasing and convex smooth function. 
The study originates from the classical Gelfand problem where $g(u)=e^u$.  

A fundamental result for (\ref{Gelfand2nd}) is that there exists $\lambda^{*}>0$, such that (c.f.~\cite{BCMR})
\begin{itemize}
  \item for $\lambda\in(0, \lambda^{*})$, there exists a {\bf classical minimal solution} $u_{\lambda}$, and by minimal we mean $u_{\lambda}\leq v$ for any other solution $v$;
  \item for $\lambda>\lambda^{*}$, there exists no solution;
  \item the map $\lambda\to u_{\lambda}(x)$ is increasing in $(0, \lambda^{*})$ for each $x\in \Omega$, thus one may define $u_{\lambda^{*}}:=\lim_{\lambda\to \lambda^{*}} u_{\lambda}$ which is shown to be the unique weak solution of ($G_{\lambda^{*}}$).
\end{itemize}
A question that has attracted a lot of research interest is the regularity of $u_{\lambda^{*}}$ and whether it is a classical solution of ($G_{\lambda^{*}}$), see e.g. ~\cite{CC}.

Now we specialize that $\Omega=B_1$, and consider the scalar curvature equation, i.e, $g(u)=e^{2u}$ if $n=2$, and $g(u)=(1+u)^{\frac{n+2}{n-2}}$, if $n\geq 3$. Let $g_0$ be the Euclidean metric on $B_1$, if $u$ is a solution to (\ref{Gelfand2nd}), then $g=e^{2u}g_0$ has Gaussian curvature $\lambda$ ($n=2$) and $g=(1+u)^{\frac{4}{n-2}}g_0$ has scalar curvature $\frac{4(n-1)}{n-2} \lambda$ ($n\geq 3$) respectively. Moreover the restriction of $g$ on the boundary coincides with the standard $S^{n-1}$. Geometrically, large spherical caps and small spherical caps account for the phenomenon of multiple solutions (when $\lambda<\lambda^{*}$). Hence it is 
alluring to think the {\bf extremal} $\lambda^{*}$ occurs whenever the solution corresponds exactly to the hemisphere. This is indeed the case in the sense that  Hang and Wang~\cite{HW} proved the following hemisphere rigidity theorem.

\begin{thm}[Hang-Wang] Let $g$ be a metric on $S^n_+$ conformal to the standard spherical metric $g_0$. Suppose
\begin{itemize}
  \item $R_g\geq n(n-1)$,
  \item the induced metric on $\partial S^{n}_{+}$ coincides with $g_0$,
\end{itemize}
then $(S^{n}_{+},g)$ is isometric to the standard hemisphere.
\end{thm}
Fitting the result of above theorem into the framework of (\ref{Gelfand2nd}) with $\Omega=B_1$, we can easily deduce that 
\begin{itemize}
    \item if $n=2$, $g(u)=e^{2u}$, then $\lambda^{*}=1$;
    \item if $n\geq 3$, $g(u)=(1+u)^{\frac{n+2}{n-2}}$, then $\lambda^{*}=\frac{(n-2)n}{4}$.
\end{itemize}

Note that Martel~\cite{M} showed that any weak supersolution of ($G_{\lambda^{*}}$) must be the weak solution. This can be viewed as an analytical counterpart of the hemisphere rigidity, as $R\geq n(n-1)$ is equivalent to the conformal factor being a supersolution. 

In this paper, we wish to generalize the hemisphere rigidity theorem to fourth order, namely under the assumption that $Q$ curvature is bounded from below,  and try to fit the result into the framework of Gelfand problem for bi-Laplace operator. 

$Q$ curvature is a fourth order curvature first introduced by Branson~\cite{Br}. There is a conformal covariant operator $P$, the Paneitz operator, associated with the $Q$ curvature. The pair $(P,Q)$ plays an important role in conformal geometry. Even though $Q$ curvature is a higher order curvature, it still reinforces considerable control on the geometry and topology of underlying manifolds, especially in dimension four (c.f. ~\cite{C}). 

Now let's give detailed definitions of $P$ and $Q$. Let $(M, g)$ be a smooth Riemannian manifold with dimension $n \geq 3,$ the $Q$ curvature is given by (following notations in~\cite{HY})
\begin{align} \notag
Q =-\Delta J-2|A|^{2}+\frac{n}{2} J^{2},
\end{align}
where
\[
J=\frac{R}{2(n-1)}, \quad A=\frac{1}{n-2}(Ric-J g).
\]
The Paneitz operator is defined as
\begin{align} \label{P}
P \varphi = \Delta^{2} \varphi+\operatorname{div}\left(4 A\left(\nabla \varphi, e_{i}\right) e_{i}-(n-2) J \nabla \varphi\right)+\frac{n-4}{2} Q \varphi,
\end{align} where $e_1, \cdots, e_n$ is an orthonormal frame with respect to $g$.

In dimension $4,$ the Paneitz operator satisfies
\[
P_{e^{2 w} g} \varphi=e^{-4 w} P_{g} \varphi,
\]
and the $Q$ curvature transforms as
\begin{align}  \label{Q4}
Q_{e^{2 w} g}=e^{-4 w}\left(P_{g} w+Q_{g}\right).
\end{align}
In dimension $n \neq 4,$ the operator satisfies
\[
P_{\rho^{ \frac{4}{n-4}} g} \varphi=\rho^{-\frac{n+4}{n-4}} P_{g}(\rho \varphi),
\]
for any positive smooth function $\rho .$  As a consequence of (\ref{P}) and above transformation law, we have
\begin{align} \label{Q}
Q_{\rho^{\frac{4}{n-4}} g}=\frac{2}{n-4} P_{\rho^{\frac{4}{n-4}} g} 1=\frac{2}{n-4} \rho^{-\frac{n+4}{n-4}} P_{g} \rho.
\end{align}

Our main theorems are 
\begin{thmx}\label{T2}
Let $g$ be a smooth metric on $S^{n}_{+}$, $n\geq3$, which is conformal to the standard round metric $g_0$. Suppose
\begin{enumerate}
  \item $Q_g\geq Q_0$,
  \item the induced metric on $\partial S^{n}_{+}$ coincides with $g_0$,
  \item $\partial S^{n}_{+}$ is totally geodesic,
  \item $R_g\geq n(n-1)$ on $\partial S^{n}_{+}$.
\end{enumerate}
Then $g=g_0$.
\end{thmx}

On a compact $4$-manifold $M$ with boundary, the Gauss-Bonnet-Chern formula is
\begin{align}\notag
4\pi^2\chi(M) =\int_{M}(\frac{|\mathcal{W}|^2}{8}+ Q )dv + \int_{\partial M} (\mathcal{L}+T) ds,
\end{align}
where $|\mathcal{W}|^2 dv$ and $\mathcal{L}ds$ are pointwise conformal invariants and $T$ is a boundary curvature term first introduced by Chang and Qing in their study of zeta functional determinants on manifolds with boundary~\cite{CQ}. As a consequence
\[
\int_{M} Q dv+ \int_{\partial M} T ds
\]
is conformally invariant. In the special case that $M$ has totally geodesic boundary (c.f.~\cite{C})
\[
T=\frac{1}{12}\frac{\partial}{\partial \nu} R,
\]
where $\nu$ is the unit outward normal of the boundary w.r.t $g$.
Using the $T$ curvature, we have 

\begin{thmx}\label{T1}
Let $g$ be a smooth metric on $S^{4}_{+}$, which is conformal to the standard round metric $g_0$. Suppose
\begin{enumerate}
  \item $Q_g\geq Q_0$,
  \item the induced metric on $\partial S^{n}_{+}$ coincides with $g_0$,
  \item $\partial S^{n}_{+}$ is totally geodesic,
  \item $\int_{\partial S^4_{+}} T d\sigma \geq0$.
\end{enumerate}
Then $g=g_0$.
\end{thmx}

There are several generalizations of Hang-Wang's theorem on locally conformlly flat manifolds, e.g., ~\cite{R},~\cite{S}. More recently, Barbosa-Cavalcante-Espinar~\cite{BCE} proved a fully nonlinear version of Hang-Wang's theorem, where the scalar curvature is replaced by some fully nonlinear operator on eigenvalues of the Schouten tensor. Our work explores its higher order extension. 

From analytical point of view, such hemisphere rigidity phenomenon hinges on the classification of the entire solution of
\[
\Delta^2 u= e^{4u}  \quad \text{for} \quad n=4; \quad \Delta^2 u=u^{\frac{n+4}{n-4}} \quad \text{for}\quad n\neq 5,
\]
which was obtained by Chang-Yang~\cite{CY}, Lin~\cite{L}, Choi-Xu~\cite{CX}, Wei-Xu~\cite{WX}. 

There is a parallel study of Gelfand problem for bi-Lapalce operator on general domain $\Omega\subset \mathbb{R}^n$, see e.g.~\cite{AGGM,DDGM}. There are usually two sets of boundary conditions, namely the Dirichlet boundary condition and the Navier boundary condition.
\begin{align}  
 \begin{cases}\Delta^{2} u=\lambda f(u) & \text { in } \Omega \\ u=\partial_{\nu} u=0 & \text { on } \partial \Omega;\end{cases} \tag{$D_{\lambda}$}
\end{align}
\begin{align}  \tag{$N_{\lambda}$}
 \begin{cases}\Delta^{2} u=\lambda f(u) & \text { in } \Omega \\ u=\Delta u=0 & \text { on } \partial \Omega.\end{cases}
\end{align}

Many of the properties of ($G_{\lambda}$) hold for Navier boundary condition (c.f.~\cite{FG}). A fundamental difficulty for Dirichlet boundary condition on general domain is the lack of maximum principle. However, thanks to Boggio's maximum principle~\cite{Bo}, when $\Omega=B_1$, similar properties for $(D_{\lambda})$ hold: (c.f.~\cite{W}) 

There exists an extremal $\lambda^{*}<\infty$, such that 
\begin{itemize}
  \item for $\lambda\in(0, \lambda^{*})$, there exists a {\bf classical minimal solution} $u_{\lambda}$;
  \item for $\lambda>\lambda^{*}$, there exists no solution;
  \item the map $\lambda\to u_{\lambda}(x)$ is increasing in $(0, \lambda^{*})$ for each $x\in \Omega$, and $u_{\lambda^{*}}:=\lim_{\lambda\to \lambda^{*}} u_{\lambda}$ is a weak solution of ($D_{\lambda^{*}}$).
\end{itemize}

Our main theorems essentially study supersolutions of ($D_{\lambda}$) under an extra condition. It is either scalar curvature lower bound or $T$ curvature lower bound, which can be interpreted as a second or a third order boundary condition on the conformal factor respectively. What we could verify at this moment is that the solution corresponding to the standard round metric is within the branch of minimal solutions of $(D_{\lambda})$. In fact, we make some speculations on the extremal $\lambda^{*}$. It is seen that $\lambda_0 \in(0, \lambda^{*})$, where $\lambda_0$ is the constant corresponding to the $Q$ curvature of the spherical metric. Hence extra assumptions may well be needed in our theorems. 

The organization of the paper is as follows: in Section~\ref{S2}, we prove Theorem~\ref{T2} and Theorem~\ref{T1}. In Section~\ref{S3}, we discuss the connection of our main results with the fourth order Gelfand problem.
%%%%%%
\vspace{0.2in}
\section{Proofs of main theorems} \label{S2}

In this section, we prove Theorem~\ref{T2} and Theorem~\ref{T1}. 

Since the standard hemisphere is conformal to the Euclidean ball, we write $g=u^{\frac{4}{n-4}} g_0$ ($n\neq 4$) and $g=e^{2u} g_0$ ($n=4)$ where $g_0$ is the Euclidean metric. Then by (\ref{Q}), $u$ satisfies 

\textbf{1. $n\geq 5$}
 \begin{align}\label{equ5}
\begin{cases}
\Delta^{2}u\geq Qu^{\frac{n+4}{n-4}} & \text{in} \quad B_{1}\\
u=u_{s}=1 & \text{on} \quad \partial B_{1}\\
\frac{\partial u}{\partial \nu}=\frac{\partial u_{s}}{\partial \nu}=-\frac{n-4}{2} &\text{on} \quad  \partial B_{1},
\end{cases}
\end{align}
where
$Q=\frac{(n-4)(n-2)n(n+2)}{16}$,
$u_{s}=(\frac{2}{1+|x|^{2}})^{\frac{n-4}{2}}$ and $\nu$ is the unit outer normal vector on $\partial B_1$.

\textbf{2. $n=4$}
 \begin{align}\label{equ6}
\begin{cases}
\Delta^{2}u  \geq 6e^{4u} & \text{in} \quad B_{1}\\
u=u_s=0 & \text{on} \quad \partial B_{1}\\
\frac{\partial u}{\partial \nu}=\frac{\partial u_s}{\partial \nu}=-1 &\text{on} \quad  \partial B_{1}.
\end{cases}
\end{align}
Here $u_s=\ln (\frac{2}{1+|x|^2})$.

\textbf{3. $n=3$}
 \begin{align}\label{equ7}
\begin{cases}
\Delta^{2}u+ \frac{15}{16}u^{-7} \leq 0 & \text{in} \quad B_{1}\\
u=1 & \text{on} \quad \partial B_{1}\\
\frac{\partial u}{\partial \nu}=\frac{1}{2} &\text{on} \quad  \partial B_{1}.
\end{cases}
\end{align}
Here $u_s=\sqrt{\frac{1+|x|^2}{2}}$.

Note $u_s$ above are conformal factors of the standard metric on the hemisphere in different dimensions. 

\begin{proof}[Proof of Theorem~\ref{T2}]\ 

The idea of the proof is as follows. By assumption, the conformal factor can be interpretted as a supersolution (subsolution for $n=3$) of a fourth order semilinear elliptic equation. We show that after a spherical average, it is still a supersolution (subsolution for $n=3$). We then use the boundary conditions and the maximum principle to get the conclusion. Since the proofs are similar, we omit most of routine proofs and only indicate essential points in dimension $4$ and $3$.

{\bf Case 1: $n\geq 5$.}

Let $v(t)=\frac{1}{|S^{n-1}(t)|} \int_{S^{n-1}(t)} u d\sigma:=\overline{u}$ be the spherical average of $u$, using the well-known formula 
\[
\overline{\Delta u}=\Delta \overline{u}
\] and the Jensen inequality, $v$ satisfies (\ref{equ5}). Note $v$ is radially symmetric. 

We {\bf claim} $v\geq u_s$. To this end, we write $u_{s}=u_{0}$, $v=u_1$, and let $u_2=tu_0+(1-t)u_1$. we construct a sequence of functions $u_{k}$ $(k\geq 3$) satisfying
$$
\begin{cases}\Delta^{2} u_{k}=Q u_{k-1}^{\frac{n+4}{n-4}} & \text { in } B_{1} \\ u_{k}=1 & \text { on } \quad \partial B_{1} \\ \frac{\partial u_{k}}{\partial n}=-\frac{n-4}{2} & \text { on } \quad \partial B_{1}.\end{cases}
$$
Note
$$
\Delta^{2} u_{3}=Q\left(t u_{0}+(1-t) u_{1}\right)^{\frac{n+4}{n-4}} \leq Q\left(t u_{0}^{\frac{n+4}{n-4}}+(1-t) u_{1}^{\frac{n+4}{n-4}}\right)\leq \Delta^{2} \left(tu_0+(1-t)u_1\right).
$$
It follows from the comparison theorem (Theorem $5.6$ in ~\cite{GGS}) that
$$
u_{3} \leq u_{2} .
$$
Inductively, we have a monotone decreasing sequence of functions
$$
u_{2} \geq u_{3} \geq \cdots.
$$
Clearly we have a uniform lower barrier
$$
u_{k}(x) \geq \frac{n}{4}-\frac{n-4}{4}|x|^{2} .
$$
Using the representation of Green's function for bi-Laplace equation and monotone convergence theorem, there exists a limit
$$
\tilde{u}(x)=\lim _{k \rightarrow \infty} u_{k}(x),
$$
which solves 
\begin{align} \notag
\begin{cases}
\Delta^{2}u=Qu^{\frac{n+4}{n-4}} & \text{in}\quad B_{1}\\
u=1 & \text{on} \quad \partial B_{1}\\
\frac{\partial u}{\partial n}=-\frac{n-4}{2} & \text{on} \quad \partial B_{1}.\end{cases}
\end{align}
According to Proposition~\ref{P2} (in the next section), $u_s$ is the minimal solution, and thus 
$v\geq u_s$.

Then 
\[
\Delta^2 v-\Delta^2 u_s\geq Q v^{\frac{n+4}{n-4}}-Qu_s^{\frac{n+4}{n-4}}\geq 0, \quad \text{in $B_1$.}
\]
Appealing to the conformal change of the scalar curvature:
\[
\frac{4(n-1)}{n-2} \Delta(u^{\frac{n-2}{n-4}}) + R_g u^{\frac{n+2}{n-4}}=0,
\]
we have that $\Delta u(1)\leq -\frac{(n+2)(n-4)}{4}$, consequently
$\Delta v(1)\leq -\frac{(n+2)(n-4)}{4}=\Delta u_s(1)$. 
It follows from the maximum principle that $\Delta (v-u_s)\leq 0$ in $B_1$, and $v-u_s=\frac{\partial v-u_s}{\partial \nu}=0$ on $\partial B_1$. It follows from the Hopf lemma that $v\equiv u_s$, and consequently that $u\equiv u_s$.

{\bf Case 2: $n=4$.} This case is not much different from case 1. The essential point is that the nonlinearity $f(u)=e^{4u}$ is increasing and convex. 

{\bf Case 3: $n=3$.} We just need to reverse directions of some inequalities: 1. $v\leq u_s$; 2. $\Delta v(1)\geq \Delta u_s(1)=\frac{5}{4}$. 

\end{proof}

For Theorem~\ref{T1}, We first establish a uniqueness result for a fourth order differential inequality. The theorem then follows from a similar spherical average method.

\begin{prop} \label{P4} Let $v: [0,\infty) \to \mathbb{R}$ be a smooth function satisfying
\begin{align} \label{2}
v''''(t)-4v''(t)\geq 6e^{4v}
\end{align}
with $v(0)=0$, $v'(0)=0$, $v'''(0)\geq 0$ and $\lim_{t\to \infty} v'(t)=-1$. Then $v=-\ln \cosh(t)$.
\end{prop}

\begin{proof} Since
\[
\left( v'''(t)-4v'(t)\right)'\geq 0,
\]
and  $v'''(0)-4v'(0)\geq 0$, we infer that $v'''(t)-4v'(t)\geq 0$ for all $t\in \mathbb{R}^+$. By the maximum principle, $v'(t)$ cannot attain any positive local maximum. In view of $v'(0)=0$ and $\lim_{t\to \infty} v'(t)=-1$, it follows that $v'(t)\leq 0$, $\forall t\in \mathbb{R}^+$. Consequently $v''(0)\leq 0$.

Multiplying $v'(t)$ to (\ref{2}), we find
\[
(v'''v'-\frac{1}{2}v''^2-2 v'^2-\frac{3}{2} e^{4v})'\leq 0.
\]
Hence
\[
(v'''v'-\frac{1}{2}v''^2-2 v'^2-\frac{3}{2} e^{4v})(0)\geq \lim_{t\to \infty} (v'''v'-\frac{1}{2}v''^2-2 v'^2-\frac{3}{2} e^{4v})(t).
\]
It follows that $v''(0)^2\leq 1$.

We divide the discussion into two cases:

\textbf{Case 1: $v''(0)=-1$.}\

We then have
\[
(v'''v'-\frac{1}{2}v''^2-2 v'^2-\frac{3}{2} e^{4v})(0)= \lim_{t\to \infty} (v'''v'-\frac{1}{2}v''^2-2 v'^2-\frac{3}{2} e^{4v})(t).
\]
Thus
\[
(v'''v'-\frac{1}{2}v''^2-2 v'^2-\frac{3}{2} e^{4v})'\equiv 0,
\] and consequently
\[
v''''(t)-4v''(t)=6e^{4v(t)}.
\]

Note $w:=-\ln \cosh(t)$ is a global solution with initial value $w(0)=w'(0)=w'''(0)=0$ and $w''(0)=-1$. We {\bf claim} $v\equiv w$. 

It suffices to rule out the possibility that $v'''(0)>0$. 
Suppose on the contrary that $v'''(0)>0$. Then $(v-w)''(t)>0$ for $t$ small. Let $T$ be the first zero (if there were)of $(v-w)''$, i.e.,   $(v-w)''(t)>0$ for $t\in(0, T)$, and $(v-w)''(T)=0$.
Consequently $(v-w)'(t)>0$, $(v-w)(t)>0$ for $t\in (0, T)$. Hence 
\[
(v-w)''''-4(v-w)''=6e^{4v}-6e^{4w} \geq 0, \quad t\in[0,T]. 
\]
A contradiction arises by looking at an interior local maximum of $(v-w)''$ in $[0,T]$. Hence $(v-w)''(t)>0$, for any $t>0$. 
This then contradicts the fact that $(v-w)'(0)=0$ and $(v-w)'(\infty)=0$.

\textbf{Case 2: $-1<v''(0)\leq 0$.} \

We shall show this case does not occur.  Note we have $v(0)=w(0)=v'(0)=w'(0)$, $v''(0)>w''(0)$ and $v'''(0)\geq w'''(0)$. 
Thus $(v-w)''(t)>0$ for $t$ small. Let $T$ be the first zero (if there were) of $(v-w)''$, i.e., $(v-w)''(t)>0$ for $t\in(0, T)$, and $(v-w)''(T)=0$. 
Consequently $(v-w)'(t)>0$, $(v-w)(t)>0$ for $t\in (0, T)$.
Hence 
\[
(v-w)''''-4(v-w)''\geq 6e^{4v}-6e^{4w} \geq 0, \quad t\in[0,T]. 
\]
A contradiction arises by looking at an interior local maximum of $(v-w)''$ in $[0,T]$. Hence $(v-w)''(t)>0$, for any $t>0$. 
This then contradicts the fact that $(v-w)'(0)=0$ and $(v-w)'(\infty)=0$.
\end{proof}

\begin{proof}[Proof of Theorem~\ref{T1}]
Under the cylindrical coordinates, the standard metric is
\[
g_0= e^{2w}\left( (dt)^2+ g_{S^3}\right),
\]
with $w=-\ln \cosh(t)$.
Note $t\in[0,\infty)$ and $\partial S^{4}_{+}$ corresponds to $t=0$.
Assume $g=e^{2u} \left( (dt)^2+ g_{S^3}\right)$. By (\ref{Q4}), the condition $Q_g\geq Q_{g_0}$ is equivalent to
\begin{align} \label{1}
(\frac{\partial^2 }{\partial t^2}+\Delta_{S^3})^2 (u)-4\frac{\partial^2 u}{\partial t^2}\geq 6 e^{4u}.
\end{align}

We consider the spherical average of $u$:
\[
v(t)=\frac{1}{\rm v_3} \int_{S^3} u(t, x) dx,\]
where ${\rm v_3}={\rm vol}(S^3)$, the volume of standard $3$-sphere.
It follows that
\begin{align}\notag
v''''(t)&=\frac{1}{\rm v_3}\int_{S^3} \frac{\partial^4 u}{\partial t^4}(t, x) dx \\ \notag
        &\geq \frac{1}{\rm v_3} \int_{S^3} -(\Delta_{S^3})^2(u)-2 \Delta_{S^3} (\frac{\partial^2 u}{\partial t^2})+4\frac{\partial^2 u}{\partial t^2} +6 e^{4u} dx \\ \notag
        &=4 v''(t)+\frac{6}{\rm v_3}\int_{S^3} e^{4u} dx \\ \label{6}
        &\geq 4 v''(t)+ 6e^{4v}.
\end{align}
We have used (\ref{1}) and the Jensen's inequality $\frac{1}{\rm v_3}\int_{S^3} e^{4u} dx \geq e^{\frac{1}{\rm v_3} \int_{S^3} 4u dx}$. Hence $v(t)$ satisfies the differential inequality (\ref{2}).

We next verify the boundary conditions of $v$. In view of boundary conditions of the metric $g$, we have
\[
u(0, x)\equiv 0 \quad \text{and} \quad \frac{\partial u}{\partial t}(0,x)\equiv 0.
\]
Thus $v(0)=0$ and $v'(0)=0$. Moreover since $t=\infty$ corresponds to a smooth interior point, it follows that $v'(t)\to -1$ as $t\to \infty$. Consequently $v''(t)\to 0$ and $v'''(t)\to 0$ as $t\to \infty$.

By the transformation law of the scalar curvature, we have
\[
6\Delta_0 (e^u)+ R_{g} e^{3u} = e^u.
\]
Note that the unit normal $\nu$ is the same as the unit normal w.r.t $g_0$, which coincides with $-\frac{\partial}{\partial t}$ on the boundary. Also note $\Delta_0$ is w.r.t the product metric $(dt)^2+ g_{S^3}$. Taking $u=0$ and $\frac{\partial u}{\partial \nu}=0$ on the boundary into consideration, we obtain that $\frac{\partial R_g}{\partial \nu}=6\frac{\partial^3 u}{\partial t^3}$.
According to the definition of $T$, we find that 
\[
v'''(0)=\frac{1}{\rm v_3} \int_{S^3} \frac{\partial ^3 u}{\partial t^3}(x, 0) dx= 2 \int_{S^3 } T dx\geq0.
\]

By Proposition~\ref{P4}, $v\equiv w$. Hence all inequalities in (\ref{6}) are equalities. It follows that $u$ is radially symmetric and thus $u=v=w$.
\end{proof}

%
%%%%%%
\vspace{0.2in}
\section{Further results and discussions} \label{S3}

In this section, we discuss some connections of our main theorems with the Gelfand problem for bi-Laplace operator. Since $n=3$ leads to a different picture, we shall assume from now on that $n\geq 4$. Recall first the fundamental properties analogous to the second order Gelfand problem. 

\begin{thm}[Warnault~\cite{W}] \label{T3}
For the Dirichlet boundary value problem ($D_{\lambda}$)
\begin{align}  
 \begin{cases}\Delta^{2} u=\lambda f(u) & \text { in } \Omega \\ u=\partial_{\nu} u=0 & \text { on } \partial \Omega,\end{cases} 
 \tag{$D_{\lambda}$}
\end{align}
where $f$ is a nonlinearity satisfying 
\begin{align} \tag{$\star$}
f \quad \text{nondecreasing}, \quad f(0)>0, \quad \text{and} \lim_{t\to\infty} \frac{f(t)}{t} =\infty,
\end{align}
there exists an extremal $\lambda^{*}<\infty$, such that 
\begin{itemize}
  \item for $\lambda\in(0, \lambda^{*})$, there exists a {\bf classical minimal solution} $u_{\lambda}$;
  \item for $\lambda>\lambda^{*}$, there exists no solution;
  \item the map $\lambda\to u_{\lambda}(x)$ is increasing in $(0, \lambda^{*})$ for each $x\in \Omega$, and $u_{\lambda^{*}}:=\lim_{\lambda\to \lambda^{*}} u_{\lambda}$ is a weak solution of ($D_{\lambda^{*}}$).
\end{itemize}
\end{thm}

We are concerned with (\ref{equ5}) and (\ref{equ6}). Even though we are dealing with non zero boundary conditions, simply by adding a linear function we may get a modified nonlinearity $f$ on the right hand satisfying ($\star$) and have Dirichlet boundary condition. The same conclusion of Theorem~\ref{T3} holds. Hence for simplicity, we may just regard $f(u)=e^{4u}$ when $n=4$ and $f(u)=(1+u)^{\frac{n+4}{n-4}}$ when $n\geq 5$.

As pointed out in the Introduction, the hemisphere rigidity theorem can yield extreme $\lambda^{*}$. It is worth to point out that in ~\cite{DDGM}, it is shown that any  weak supersolution of $(D_{\lambda^{*}})$ must be the solution. Our main theorems require an extra assumption on the lower bound of either the scalar curvature or the $T$ curvature on the boundary, which is a second order or third order boundary condition for the conformal factor. Nevertheless, the conformal factor for the standard spherical metric in within the branch of {\bf classical minimal solutions}.  

\begin{prop} \label{P2}
Let $Q=\frac{(n-4)(n-2)n(n+2)}{16}$ and suppose $n\geq 5$. There are at most \textbf{two} positive solutions satisfying
\begin{align} \label{8}
\begin{cases}
\Delta^{2}u=Qu^{\frac{n+4}{n-4}} & \text{in}\quad B_{1}\\
u=1 & \text{on} \quad \partial B_{1}\\
\frac{\partial u}{\partial n}=-\frac{n-4}{2} & \text{on} \quad \partial B_{1}.\end{cases}
\end{align}
Moreover, $u_s=(\frac{2}{1+|x|^2})^{\frac{n-4}{2}}$ is the minimal solution. 
\end{prop}

\begin{proof}
By the work of Berchio-Gazzola-Weth~\cite{BGW}, the solution of (\ref{8}) is radial symmetric.
The Pohozaev identity for the radial solution of $\Delta^{2}u=Qu^{\frac{n+4}{n-4}}$ (c.f. ~\cite{GOR}) implies that at $r=1$
\begin{align*}
0=r^{n-1}(\Delta u)^{\prime}\left(ru^{\prime}+\frac{n-4}{2}u\right)+\frac{n}{2}r^{n-1}u^{\prime}\Delta u-r^{n}\left(\frac{1}{2}(\Delta u)^{2}+\frac{n-4}{2n}Qu^{\frac{2n}{n-4}}\right).
\end{align*}
Plugging the boundary conditions, we infer that
\begin{align} \label{poho}
\frac{n}{2}u^{\prime}(1)\Delta u(1)-\left(\frac{1}{2}(\Delta u(1))^{2}+\frac{n-4}{2n}Qu(1)^{\frac{2n}{n-4}}\right)=0,
\end{align}
from which we find two possible values for $\Delta u(1)$:
\[
a_1=-\frac{(n-4)(n+2)}{4}, a_2=-\frac{(n-4)(n-2)}{4}.
\]

Next we show the solution of 
\begin{align} \label{eq:three boundary condition}
\begin{cases}
\Delta^{2}u=Qu^{\frac{n+4}{n-4}} & \text{in}\quad B_{1}\\
u=c_1 & \text{on} \quad \partial B_{1}\\
\frac{\partial u}{\partial n}=c_2 & \text{on} \quad \partial B_{1}\\
\Delta u=c_3 &\text{on} \quad \partial B_{1}.
\end{cases}
\end{align}
is unique. 

Suppose there are two solutions $u_{1}$, $u_{2}$ of
(\ref{eq:three boundary condition}).
We \textbf{claim} that $u_{1}$
and $u_{2}$ are ordered. If not, first of all, as radial solutions, the zeros of $u_1-u_2$ do not accumulate, let $r_{*}$ be the largest zero of $u_1-u_2$ in $(0,1)$. Without loss of generality, we assume that $u_{1}(r)\ge u_{2}(r)$ on $[r_{*},1]$.
Setting $w=u_{1}-u_{2}$, then we have
\[ 
w(r_{*})=0,\quad w(r)> 0 \quad \quad \forall r\in (r_{*},1).
\]
Let $r_{**}\in(r_{*},1)$ be such that
\[
(u_{1}-u_{2})(r_{**})=\max_{[r_{*},1]}(u_{1}-u_{2}).
\]
It follows that $w'(r_{**})=0$ and $\Delta w(r_{**})\le0$.
On $B_{1}\backslash B_{r_{**}},$
\[
\Delta^{2}(u_{1}-u_{2})=Q(u_{1}^{\frac{n+4}{n-4}}-u_{2}^{\frac{n+4}{n-4}})\ge0.
\]
By the maximum principle,
\[
\Delta w\le\max_{\partial(B_{1}\backslash B_{r_{**}})}\Delta w\le0\quad on\quad B_{1}\backslash B_{r_{**}},
\]
and then
\[
w\ge\min_{\partial(B_{1}\backslash B_{r_{**}})}w=\min\{w|_{\partial B_{1}},w|_{\partial B_{r_{**}}}\}.
\]
But we know $\frac{\partial w}{\partial n}|_{\partial B_{r_{**}}}=0=\frac{\partial w}{\partial n}|_{\partial B_{1}},$
which contradicts to the Hopf Lemma. The claim thus is proved.

Without loss of generality, we assume $u_{1}\ge u_{2}$ in $B_{1}$. Thus
we have $\Delta^{2}u_{1}\ge\Delta^{2}u_{2}$ in $B_{1}.$ Since $\Delta u_{1}=\Delta u_{2}$
on $\partial B_{1}$, we have $\Delta(u_{1}-u_{2})\le0$ in $B_{1}$. But
$\frac{\partial(u_{1}-u_{2})}{\partial n}=0$ on $\partial B_{1}$
which contradicts to the Hopf Lemma again.

Finally, note $\Delta u_s(1)=a_1$. Let $u_2$ be another solution (if exists) corresponds to $\Delta u_2(1)=a_2$. We shall show that $u_2\geq u_s$. Write $u_{s}=u_{1}$ for simplicity, and we construct a sequence of functions $u_{k}$ satisfying
$$
\begin{cases}\Delta^{2} u_{k}=Q u_{k-1}^{\frac{n+4}{n-4}} & \text { in } B_{1} \\ u_{k}=1 & \text { on } \quad \partial B_{1} \\ \frac{\partial u_{k}}{\partial n}=-\frac{n-4}{2} & \text { on } \quad \partial B_{1},\end{cases}
$$
with $u_{3}=t u_{1}+(1-t) u_{2}$. Notice
$$
\Delta^{2} u_{4}=Q\left(t u_{1}+(1-t) u_{2}\right)^{\frac{n+4}{n-4}} \leq Q\left(t u_{1}^{\frac{n+4}{n-4}}+(1-t) u_{2}^{\frac{n+4}{n-4}}\right)=\Delta^{2} u_{3} .
$$
It follows from the comparison theorem (Theorem $5.6$ in ~\cite{GGS}) that
$$
u_{4} \leq u_{3} .
$$
Inductively, we have a monotone decreasing sequence of functions
$$
u_{3} \geq u_{4} \geq \cdots.
$$
Clearly we have a uniform lower barrier
$$
u_{k}(x) \geq \frac{n}{4}-\frac{n-4}{4}|x|^{2} .
$$
Using the representation of Green's function for bi-Laplace equation and monotone convergence theorem, there exists a limit
$$
\tilde{u}(x)=\lim _{k \rightarrow \infty} u_{k}(x),
$$
which solves (10). Therefore we have $t u_{1}+(1-t) u_{2} \geq u_{1}$ or $t u_{1}+(1-t) u_{2} \geq u_{2}$. It follows $u_{2} \geq u_{1}$ since $\Delta u_{1}(1)<\Delta u_{2}(1)$.
\end{proof}

Now for
\begin{align}  \tag{$P_{\lambda}$}
\begin{cases}
\Delta^{2}u=\lambda u^{\frac{n+4}{n-4}} & \text{in}\quad B_{1}\\
u=1 & \text{on} \quad \partial B_{1}\\
\frac{\partial u}{\partial n}=-\frac{n-4}{2} & \text{on} \quad \partial B_{1},\end{cases}
\end{align}
appealing to the Pohozaev identity (\ref{poho}) in the above proof, for $\lambda\in(0, \frac{n^3(n-4)}{16})$, there are at most two solutions for ($P_{\lambda}$); and for $\lambda=\frac{n^3(n-4)}{16}$, there is a unique value for $\Delta u(1)$, thus the solution to ($P_{\lambda^{*}}$) is unique. Clearly for $\lambda>\frac{n^3(n-4)}{16}$ there exists no solution. Hence it is tempting to make the following conjecture.
\begin{conjecture}
The extreme for ($P_{\lambda}$) is $\lambda^{*}=\frac{n^3(n-4)}{16}$.
\end{conjecture}

\begin{rem}

If this conjecture were true, then there is a hemisphere rigidity theorem, under the assumption that: 1. $Q_g\geq \frac{n^3}{8}>Q_0=\frac{(n-2)n(n+2)}{8}$, 2. $g$ coincides with $g_0$ on the boundary and 3. the boundary is totally geodesic.  It will be interesting to find a geometric explanation for $Q$ curvature lower bound corresponding to the extremal $\lambda^{*}$. 
\end{rem}

Similarly, we have 
\begin{prop} For $n=4$, there are at most \textbf{two} solutions satisfying
\begin{align} \label{9}
\begin{cases}
\Delta^{2}u=6u^{4u} & \text{in}\quad B_{1}\\
u=0 & \text{on} \quad \partial B_{1}\\
\frac{\partial u}{\partial n}=-1 & \text{on} \quad \partial B_{1}.\end{cases}
\end{align}
Moreover, $u_s=\ln(\frac{2}{1+|x|^2})$ is the minimal solution. 
\end{prop}
\begin{proof}
By the work of Berchio-Gazzola-Weth~\cite{BGW}, the solution of (\ref{9}) is radial symmetric.
Using the cylindrical coordinates, which amounts to the substitution $v(t):=u(e^{-t})-t$, then $v:[0,\infty) \to \mathbb{R}$ satisfies:
\begin{align} \label{10}
v''''(t)-4v''(t)=6e^{4v},
\end{align}
and $v(0)=v'(0)=0$.
Multiplying (\ref{10}) to $v'(t)$, we get 
\begin{align} \label{11}
\left( v'''v'-\frac{1}{2}(v'')^2-2(v')^2-\frac{3}{2}e^{4v}\right)'=0.
\end{align}
Noting that $v'(\infty)=-1$, $v''(\infty)=v'''(\infty)=0$, we infer that $v''(0)=\pm 1$. 
A similar maximum principle as in the proof of Proposition~\ref{P2} yields desired conclusions. 
\end{proof}

Appealing to the first integral identity (\ref{11}), we have 
\begin{conjecture}
For the problem 
\begin{align} \tag{$Q_{\lambda}$}
\begin{cases}
\Delta^{2}u=\lambda e^{4u} & \text{in}\quad B_{1}\\
u=1 & \text{on} \quad \partial B_{1}\\
\frac{\partial u}{\partial n}=-1 & \text{on} \quad \partial B_{1},\end{cases}
\end{align}
the extremal $\lambda^{*}=8$.
\end{conjecture}

We finally mention a stability characterization of the extremal $\lambda^{*}$, (c.f. Proposition 36 and 37 in ~\cite{AGGM}), which extends the result of the second order ~\cite{CR}. The stability is in terms of the nonnegativity of the linearized operator, and it is at the extremal $\lambda^{*}$ for which the first eigenvalue of linearized operator is zero. This is responsible for solutions to bifurcate at $\lambda^{*}$.

\begin{prop} Suppose $u_{\lambda}$ is a regular solution for $(P_{\lambda})$, $\lambda\in (0, \lambda^{*}]$,  let $\mu_1$ be the first eigenvalue of the linearized operator $\mathcal{L}_{\lambda}(\varphi):=\left( \Delta^2- \frac{n+4}{n-4} \lambda u_{\lambda}^{\frac{8}{n-4}}\right) (\varphi) $ subject to the Dirichlet boundary condition $\varphi=\partial_{\nu} \varphi=0$, then $\mu_1=0$ if and only if $\lambda=\lambda^{*}$. The same holds true for $(Q_{\lambda})$, where $\mathcal{L}_{\lambda}(\varphi):=\left( \Delta^2 -4\lambda e^{4u_{\lambda}}\right) (\varphi)$.
\end{prop}

\end{document}